\documentclass[a4paper, svgnames, 10pt]{amsart}

\RequirePackage[l2tabu, orthodox]{nag} % for proofreading
\usepackage[all]{onlyamsmath}          % for proofreading

\usepackage{
    amssymb,
    mathtools,
    mathrsfs,
    xcolor,
    tikz-cd,
    graphicx,
}

\usepackage[english]{babel}

\usepackage{hyperref}
\hypersetup{colorlinks, allcolors=MediumBlue}
\newcommand{\msc}[1]{\href{https://zbmath.org/classification/?q=#1}{#1}}
\newcommand{\doi}[1]{doi:\href{https://doi.org/#1}{#1}}
\usepackage[noabbrev, capitalise]{cleveref}
\crefdefaultlabelformat{#2\textup{#1}#3}

\crefname{main} {Theorem}       {Theorems}
\crefname{thm}  {Theorem}       {Theorems}
\crefname{lem}  {Lemma}         {Lemmas}
\crefname{prop} {Proposition}   {Propositions}
\crefname{hyp}  {Hypothesis}    {Hypotheses}
\crefname{exm}  {Example}       {Examples}

\theoremstyle{plain}
\newtheorem{main} {Theorem}

\newtheorem{thm} {Theorem} [section]
\newtheorem{prop}   [thm] {Proposition}
\newtheorem{lem}    [thm] {Lemma}
\theoremstyle{definition}
\newtheorem{hyp}    [thm] {Hypothesis}
\newtheorem{exm}    [thm] {Example}
\newtheorem*{prob}        {Problem}

\usepackage[a-1b]{pdfx}

\numberwithin{equation}{section}

\newcommand*{\case}[1]{\paragraph{\indent\textit{#1}}}

\newcommand{\CC}{\mathbb{C}}
\newcommand{\QQ}{\mathbb{Q}}
\newcommand{\ZZ}{\mathbb{Z}}
\newcommand{\totient}{\phi}
\newcommand{\idealm}{\mathfrak{m}}
\newcommand{\idealn}{\mathfrak{n}}
\newcommand{\idealr}{\mathfrak{r}}

\DeclarePairedDelimiterX\set[1]\lbrace\rbrace{\,\def\given{\mid}#1\,}
\DeclarePairedDelimiterX\gen[1]\langle\rangle{\,\def\given{\mid}#1\,}
\DeclareMathOperator{\lcm}{lcm}
\DeclareMathOperator{\Ker}{Ker}
\newcommand{\CS}{\mathscr{C}}                % the set of conjugacy classes of cyclic subgroups

\begin{document}

\title[Where isomorphisms of group algebras fail to lift]{Where isomorphisms of group algebras\\
    fail to lift}

% LM
\author[L.~Margolis]{Leo Margolis\,%\orcidlink{0000-0002-9413-0775}
}
\address[Leo Margolis]
{Universidad Aut\'onoma de Madrid, Departamento de Matem\'aticas,  C/ Francisco Tom\'as y Valiente 7, Facultad de Ciencias, m\'odulo 17, 28049 Madrid, Spain.}
\email{leo.margolis@icmat.es}

% TS
\author[T.~Sakurai]{Taro Sakurai\,%\orcidlink{0000-0003-0608-1852}
}
\address[Taro Sakurai]
{Department of Mathematics and Informatics, Graduate School of Science, Chiba University, 1-33, Yayoi-cho, Inage-ku, Chiba-shi, Chiba, 263-8522, Japan.}
\email{tsakurai@math.s.chiba-u.ac.jp}

\thanks{
    The first author acknowledges financial support from the Spanish Ministry of Science and Innovation, through the Ram\'on y Cajal grant program.
    He would also like to thank Chiba University for its hospitality.
}

% 20C05: Group rings of finite groups and their modules (group-theoretic aspects)
% 16S34: Group rings [See also 20C05, 20C07], Laurent polynomial rings
% 20D15: Finite nilpotent groups, p-groups
% 16U60: Units, groups of units (associative rings and algebras)


\date{\today}

\begin{abstract}
    Counterexamples to the Modular Isomorphism Problem were discovered recently.
    These are non-isomorphic finite $2$-groups $G$ and $H$ that have isomorphic group algebras over the field $\mathbb{Z}/2\mathbb{Z}$ and non-isomorphic group algebras over the $2$-adic integers $\mathbb{Z}_2$.
    We show that the groups $G$ and $H$ already have non-isomorphic group algebras over the ring $\mathbb{Z}/4\mathbb{Z}$.
\end{abstract}

\maketitle

%\tableofcontents

\section{Introduction}
\index{$\ZZ$ : Commutative Ring}
\index{$p$ : Prime Number}
\index{$\ZZ_p$ : Commutative Ring}
\index{$G$ : Finite Group}
\index{$H$ : Finite Group}
\index{$k$ : Positive Integer}
\index{$P$ : Finite Group}

The Integral Isomorphism Problem asks whether non-isomorphic finite groups have non-isomorphic group algebras over the integers~$\ZZ$.
This problem was proposed by Higman~\cite{Higman40} and he proved that it has a positive solution for finite abelian groups.
For a prime~$p$, Roggenkamp--Scott~\cite{RoggenkampScott87} and Weiss~\cite{Weiss88} proved that non-isomorphic finite $p$-groups have non-isomorphic group algebras even over the $p$-adic integers~$\ZZ_p$.
This result is sharp, because Hertweck~\cite{Hertweck01} discovered a counterexample among finite solvable groups whose orders are divisible by exactly two distinct primes.
In characteristic~$p$, the Modular Isomorphism Problem asks whether non-isomorphic finite $p$-groups have non-isomorphic group algebras over the field~$\ZZ/p\ZZ$.
For a survey, see \cite{Margolis22}.
Deskins~\cite{Deskins56} proved that this problem has a positive solution for finite abelian $p$-groups.
Recently Garc\'{\i}a-Lucas--Margolis--del~R\'{\i}o~\cite{GarciaLucasMargolisDelRio22} discovered counterexamples to the problem, later generalized in~\cite{BaginskiZabielski25, MargolisSakurai25}.
These are non-isomorphic finite $2$-groups $G$ and $H$ that have isomorphic group algebras over~$\ZZ/2\ZZ$.
By contrast, the groups $G$ and $H$ have non-isomorphic group algebras over~$\ZZ_2$.
Indeed, there is a Maranda-type theorem \cite[(III.5.10)]{Sehgal78} and the same is true for $\ZZ/2^k\ZZ$ when $k$ is sufficiently large.
Sehgal \cite[(III.5.11)]{Sehgal78} asked whether every isomorphism between modular group algebras lifts to an isomorphism between $p$-adic group algebras for finite $p$-groups (see also \cite[Question~4.9]{Linckelmann18} for Morita equivalences of blocks.)
The recent counterexamples answer this question in the negative, and hence there must be a minimal~$k \geq 2$ such that the groups $G$ and $H$ have non-isomorphic group algebras over $\ZZ/2^k\ZZ$.
\begin{equation*}
    \begin{tikzcd}[sep=small]
        \ZZ/2\ZZ \, G
         & \dotsb                \arrow[l, two heads]
         & \ZZ/2^{k - 1}\ZZ \, G \arrow[l, two heads]
         & \ZZ/2^k\ZZ \, G       \arrow[l, two heads]
         & \ZZ/2^{k + 1}\ZZ \, G \arrow[l, two heads]
         & \dotsb                \arrow[l, two heads]
         & \ZZ_2 \, G            \arrow[l, two heads]
        \\
        \ZZ/2\ZZ \, H                                 \arrow[u, phantom, sloped, "\cong"]
         & \dotsb                \arrow[l, two heads]
         & \ZZ/2^{k - 1}\ZZ \, H \arrow[l, two heads] \arrow[u, phantom, sloped, "\cong"]
         & \ZZ/2^k\ZZ \, H       \arrow[l, two heads] \arrow[u, phantom, sloped, "\not\cong"]
         & \ZZ/2^{k + 1}\ZZ \, H \arrow[l, two heads] \arrow[u, phantom, sloped, "\not\cong"]
         & \dotsb                \arrow[l, two heads]
         & \ZZ_2 \, H            \arrow[l, two heads] \arrow[u, phantom, sloped, "\not\cong"]
    \end{tikzcd}
\end{equation*}
In this paper, we prove that the groups $G$ and $H$ already have non-isomorphic group algebras over $\ZZ/4\ZZ$.
More generally, we prove the following.

\begin{main}\label{main:A}
    \index{$R$ : Commutative Ring}
    \index{$\idealr$ : Ideal of Ring}
    \index{$n$ : Positive Integer}
    \index{$m$ : Positive Integer}
    \index{$\ell$ : Positive Integer}
    \index{$x$ : Group Element}
    \index{$y$ : Group Element}
    \index{$z$ : Group Element}
    \index{$a$ : Group Element}
    \index{$b$ : Group Element}
    \index{$c$ : Group Element}
    \index{$\idealm$ : Ideal of Ring}

    Let $R$ be a commutative ring with an ideal $\idealr$.
    Suppose $n > m > \ell \geq 2$ and let
    \begin{alignat*}{8}
         &G &\  &= &\ \langle\, x, y, z \mid x^{2^n} &= 1, &\ y^{2^m} &= 1,       &\ z^{2^\ell} &= 1, &\ [y, x] &= z, &\ [z, x] &= z^{-2}, &\ [z, y] &= z^{-2} \,\rangle, \\
         &H &\  &= &\ \langle\, a, b, c \mid a^{2^n} &= 1, &\ b^{2^m} &= a^{2^m}, &\ c^{2^\ell} &= 1, &\ [b, a] &= c, &\ [c, a] &= c^{-2}, &\ [c, b] &= c^{-2} \,\rangle.
    \end{alignat*}
    If $R$ has a maximal ideal $\idealm$ such that $\idealm^2 \subseteq \idealr \subseteq \idealm$ and $R/\idealr$ has characteristic~$4$, then $RG \not\cong RH$.
    On the other hand, if $R$ has characteristic~$2$, then $RG \cong RH$.
\end{main}

The non-isomorphic groups $G$ and $H$ of order $2^{n + m + \ell}$ in the theorem are essentially\footnote{Taking the same power of these groups and taking the direct product with a common finite $2$-group $P$ also yields counterexamples for a trivial reason.
    So the only known counterexamples are of the form $G \times \dotsb \times G \times P$ and $H \times \dotsb \times H \times P$.
} the only known counterexamples to the Modular Isomorphism Problem.
Since our ground ring is no longer a field, we lay the groundwork for our analysis in \cref{sec:Preliminaries}.
The theorem is proved in \cref{sec:Approximation}.
Our proof is based on a group base approximation used in~\cite{MargolisSakurai25}.
In \cref{sec:Rational} we study the rational and complex cases, which are not covered in \cref{main:A}, and show that the situation is more delicate when $2$ is a unit in the ground ring.

\section{Preliminaries}\label{sec:Preliminaries}
\index{$G$ : Group}
\index{$Z(G)$ : Group}
\index{$\Phi(G)$ : Group}
\index{$n$ : Positive Integer}
\index{$\gamma_n(G)$ : Group}
\index{$g$ : Group Element}
\index{$h$ : Group Element}
\index{$[g, h]$ : Group Element}
\index{$g^h$ : Group Element}
\index{$g^G$ : Set}
\index{$\vert g \vert$ : Positive Integer}
\index{$K$ : Group}
\index{$C_G(K)$ : Group}
\index{$N_G(K)$ : Group}
\index{$C_n$ : Group}
\index{$p$ : Prime Number}
\index{$k$ : Nonnegative Integer}
\index{$G^{p^k}$ : Group}
\index{$\Omega_k(G)$ : Group}
\index{$N$ : Group}
\index{$\Omega(G : N)$ : Group}
\index{$Z(SG)$ : Algebra}
\index{$A$ : Algebra Element}
\index{$B$ : Algebra Element}
\index{$S$ : Commutative Ring}
\index{$\varepsilon_{SG}$ : Algebra Homomorphism}
\index{$\alpha_g$ : Ring Element}
\index{$\Delta(SG)$ : Ideal of Algebra}
\index{$\idealn$ : Ideal of Ring}
\index{$F$ : Commutative Ring}
\index{$\Delta(SG : \idealn)$ : Ideal of Algebra}
\index{$\pi$ : Ring Homomorphism}
\index{$\hat\pi$ : Algebra Homomorphism}

We use standard group-theoretical notation.
For a group $G$, we write its center as $Z(G)$, its derived subgroup as $G'$ and its Frattini subgroup as $\Phi(G)$.
The $n$-th term of the lower central series is denoted by $\gamma_n(G)$.
For $g, h \in G$, the commutator is defined as $[g, h] = g^{-1}h^{-1}gh$ and the conjugate is defined as $g^h = h^{-1}gh$.
The conjugacy class of $g$ in $G$ is denoted by $g^G$ and the order of $g$ is denoted by $|g|$.
For a subgroup $K \leq G$, its centralizer and normalizer in $G$ are denoted by $C_G(K)$ and $N_G(K)$.
Let $p$ be a prime and suppose that $G$ is a finite $p$-group.
For a normal subgroup $N \trianglelefteq G$, we define
\begin{equation*}
    \Omega(G : N) = \gen{ g \in G \given g^p \in N }.
\end{equation*}
For a non-negative integer $k$, we define $\Omega_k(G) = \gen{ g \in G \given g^{p^k} = 1 }$ and $G^{p^k} = \gen{ g^{p^k} \given g \in G }$.
A cyclic group of order $n$ is denoted by $C_n$.

For a commutative ring $S$ and a finite group $G$, we write the center of the group algebra $SG$ as $Z(SG)$.
For $A, B \in SG$, the Lie commutator is defined as $[A, B] = AB - BA$.
The augmentation ideal $\Delta(SG)$ of $SG$ is defined to be the kernel of a map
\begin{equation*}
    \varepsilon_{SG} \colon SG \to S, \ \ \sum_{g \in G} \alpha_g g \mapsto \sum_{g \in G} \alpha_g,
\end{equation*}
called the augmentation map.

For an ideal $\idealn$ of $S$, we set $F = S/\idealn$ and
\begin{equation*}
    \Delta(SG : \idealn)
    = \set{ x \in SG \given \varepsilon_{SG}(x) \in \idealn }
    = \idealn \oplus \Delta(SG).
\end{equation*}
Let $\pi \colon S \to F$ be the natural projection and extend it to $\hat\pi \colon SG \to FG$ in the natural manner.
Note that $\Ker \hat\pi = \idealn \oplus \idealn\Delta(SG)$.

In the rest of this section we will prove some preparatory lemmas which will then serve to study the concrete groups in the next section.
We start with a lemma which is well known, but seems not to be contained in this short and general form in the literature.

\begin{lem}\label{lem:reduce}
    \index{$R$ : Commutative Ring}
    \index{$S$ : Commutative Ring}
    \index{$G$ : Finite Group}
    \index{$H$ : Finite Group}

    Let $R$ and $S$ be commutative rings and $G$ and $H$ finite groups.
    Suppose that there is a ring homomorphism $R \to S$.
    If $RG \cong RH$, then $SG \cong SH$.
\end{lem}
\begin{proof}
    \index{$\rho$ : Ring Homomorphism}
    \index{$\varphi$ : Algebra Homomorphism}
    \index{$g$ : Group Element}
    \index{$h$ : Group Element}
    \index{$\alpha_{g, h}$ : Ring Element}
    \index{$\psi$ : Algebra Homomorphism}

    Let $\rho\colon R \to S$ be a ring homomorphism and $\varphi\colon RG \to RH$ an $R$-algebra isomorphism.
    For $g, h \in G$, take $\alpha_{g, h} \in R$ that satisfy $\varphi(g) = \sum_{h \in H} \alpha_{g, h} h$ and define $\psi\colon SG \to SH$ by $\psi(g) = \sum_{h \in H} \rho(\alpha_{g, h}) h$.
    It is straightforward to check that $\psi$ is indeed an $S$-algebra isomorphism whose inverse is defined similarly.
\end{proof}

\begin{lem}\label{lem:kernel}
    \index{$S$ : Commutative Ring}
    \index{$\idealn$ : Ideal of Ring}
    \index{$G$ : Finite Group}
    \index{$F$ : Commutative Ring}
    \index{$\pi$ : Ring Homomorphism}
    \index{$\hat\pi$ : Algebra Homomorphism}

    Let $S$ be a commutative ring with an ideal $\idealn$ and $G$ a finite group.
    Write $F = S/\idealn$, $\pi \colon S \to F$ and $\hat\pi \colon SG \to FG$ as before.
    Then
    \begin{equation*}
        \Delta(SG) \cap \hat\pi^{-1}(\Delta(FG)^2) = \idealn\Delta(SG) + \Delta(SG)^2.
    \end{equation*}
\end{lem}
\begin{proof}
    \index{$A$ : Algebra Element}
    \index{$g$ : Group Element}
    \index{$h$ : Group Element}
    \index{$\alpha_{g, h}$ : Ring Element}
    \index{$\beta_{g, h}$ : Ring Element}

    It is easy to see that the right-hand side is contained in the left-hand side.
    To see the converse, take $A \in \Delta(SG) \cap \hat\pi^{-1}(\Delta(FG)^2)$.
    Then there are some $\alpha_{g, h} \in F$ such that
    \begin{equation*}
        \hat\pi(A) = \sum_{g \in G}\sum_{h \in G} \alpha_{g, h}(g - 1)(h - 1).
    \end{equation*}
    Since $\pi$ is surjective, there are some $\beta_{g, h} \in S$ such that $\pi(\beta_{g, h}) = \alpha_{g, h}$.
    Then
    \begin{equation*}
        A - \sum_{g \in G}\sum_{h \in G} \beta_{g, h}(g - 1)(h - 1) \in \Ker \hat\pi.
    \end{equation*}
    It follows from $(g - 1)(h - 1) \in \Delta(SG)^2$ and $\Ker \hat\pi = \idealn \oplus \idealn\Delta(SG)$ that $A \in \idealn \oplus (\idealn\Delta(SG) + \Delta(SG)^2)$.
    Since $A \in \Delta(SG)$, we conclude that $A \in \idealn\Delta(SG) + \Delta(SG)^2$.
\end{proof}

\begin{prop}\label{prop:RelAug}
    \index{$S$ : Commutative Ring}
    \index{$\idealn$ : Ideal of Ring}
    \index{$G$ : Finite Group}
    \index{$F$ : Commutative Ring}

    Let $S$ be a commutative ring with an ideal $\idealn$ and $G$ a finite group.
    Write $F = S/\idealn$.
    Then there is an isomorphism of $F$-modules
    \begin{equation*}
        \Delta(SG : \idealn)/\Delta(SG : \idealn)^2 \cong \idealn/\idealn^2 \oplus \Delta(FG)/\Delta(FG)^2.
    \end{equation*}
\end{prop}
\begin{proof}
    Recall that $\Delta(SG : \idealn) = \idealn \oplus \Delta(SG)$.
    As
    \begin{align*}
        \frac{\Delta(SG : \idealn)}{\Delta(SG : \idealn)^2}
        = \frac{\idealn \oplus \Delta(SG)}{\idealn^2 \oplus (\idealn\Delta(SG) + \Delta(SG)^2)}
        \cong \frac{\idealn}{\idealn^2} \oplus \frac{\Delta(SG)}{\idealn\Delta(SG) + \Delta(SG)^2},
    \end{align*}
    it suffices to prove that
    \begin{equation*}
        \frac{\Delta(SG)}{\idealn\Delta(SG) + \Delta(SG)^2} \cong \frac{\Delta(FG)}{\Delta(FG)^2}.
    \end{equation*}
    Consider the canonical map $\Delta(SG) \to \Delta(FG)/\Delta(FG)^2$.
    By \cref{lem:kernel}, we see that the kernel of this map is $\idealn\Delta(SG) + \Delta(SG)^2$ as expected.
\end{proof}

The $n$-th dimension subgroup of a group $G$ with respect to a commutative ring $S$ is the subgroup of $G$ defined by \[D_{n,S}(G) = \set{ g \in G \given g - 1 \in \Delta(SG)^n }.
\]
When $G$ is a finite $p$-group and $S$ is a field of characteristic $p$, the following purely group-theoretical formula describing the dimension subgroups was observed by Jennings and reformulated by Lazard.

\begin{prop}[{\cite[Theorem~11.1.20]{Passman77}}]\label{prop:JenningsTheory}
    \index{$F$ : Field}
    \index{$p$ : Prime Number}
    \index{$G$ : Finite Group}
    \index{$n$ : Positive Integer}
    \index{$i$ : Positive Integer}
    \index{$j$ : Nonnegative Integer}
    \index{$\ell$ : Positive Integer}
    \index{$g_1, \dotsc, g_\ell$ : Group Element}

    Let $F$ be a field of positive characteristic~$p$ and $G$ a finite $p$-group.
    Then $D_{n,F}(G) = \prod_{ip^j \geq n} \gamma_i(G)^{p^j}$.
    In particular,
    \begin{equation*}
        \Delta(FG)/\Delta(FG)^2 \cong F \otimes_{\ZZ/p\ZZ} G/\Phi(G)
    \end{equation*}
    as $F$-vector spaces and a basis of $\Delta(FG)/\Delta(FG)^2$ is given by representatives $g_1-1$, \dots, $g_\ell-1$ where $g_1$, \dots, $g_\ell$ form a minimal generating set of $G$.
\end{prop}

This kind of compact formula is not known for dimension subgroups in the general situation, though some information is available, see e.g. \cite{Passi79, Sandling72}.

Moreover, the following is also well known, but we give a short proof here as we did not find it formulated in that compact form in the literature.

\begin{lem}\label{lem:DimensionSubgroupsAndGamma}
    \index{$S$ : Commutative Ring}
    \index{$G$ : Group}
    \index{$n$ : Positive Integer}

    Let $S$ be a commutative ring, $G$ a group and $n$ a positive integer.
    Then $\gamma_n(G) \leq D_{n,S}(G)$.
\end{lem}
\begin{proof}
    \index{$g$ : Group Element}
    \index{$h$ : Group Element}

    We argue by induction on $n \geq 1$.
    The base case is trivial.
    Let $g \in G$ and $h \in \gamma_n(G)$.
    As $h \in D_{n,S}(G)$ by the inductive hypothesis, we have $[h-1, g-1] \in \Delta(SG)^{n + 1}$.
    So
    \begin{equation*}
        [h-1, g-1] = hg - gh = hg(1-[g, h])
    \end{equation*}
    implies that $[g, h] \in D_{n+1,S}(G)$.
\end{proof}

The following easy observations will be very useful for calculations in group algebras over commutative rings of characteristic $4$.

\begin{lem}\label{lem:KummerBinomial}
    \index{$p$ : Prime Number}
    \index{$n$ : Positive Integer}
    \index{$i$ : Positive Integer}
    \index{$\nu$ : Function}

    Let $p$ be a prime, $n$ a positive integers and $0 < i < p^n$.
    Then
    \begin{equation*}
        \nu\left(\binom{p^n}{i}\right) = (p-1)(n - \nu(i)),
    \end{equation*}
    where $\nu$ denotes the $p$-adic valuation.
\end{lem}
\begin{proof}
    \index{$k$ : Nonnegative Integer}
    \index{$i_1, \dotsc, i_{n - 1}$ : Nonnegative Integer}

    Let $k = \nu(i)$ and expand $i$ as
    \begin{equation*}
        i = i_k p^k + i_{k + 1} p^{k + 1} + \dotsb + i_{n - 1} p^{n - 1},
    \end{equation*}
    where $0 \leq i_k, i_{k + 1}, \dotsc, i_{n - 1} < p$ and $i_k \neq 0$.
    Since
    \[p^n - 1 = (p-1)(1 + p + \dots + p^{n - 1}) = p^k-1 + (p-1)p^k + \dots + (p-1)p^{n-1}, \]
    we obtain
    \begin{equation*}
        p^n - i = (p - i_k)p^k + (p-1 - i_{k + 1})p^{k + 1} + \dotsb + (p-1 - i_{n - 1})p^{n - 1}.
    \end{equation*}
    It follows from Kummer's theorem~\cite[p.~116]{Kummer1852} on binomial coefficients that
    \begin{align*}
        \nu\left(\binom{p^n}{i}\right)
         &= (i_k + i_{k + 1} + \dotsb + i_{n - 1})                                  \\
         &\qquad + ((p - i_k) + (p-1 - i_{k + 1}) + \dotsb + (p-1 - i_{n - 1})) - 1 \\
         &= (p-1)(n - k).
        \qedhere
    \end{align*}
\end{proof}

\begin{lem}\label{lem:KummerMod4}
    \index{$S$ : Commutative Ring}
    \index{$G$ : Group}
    \index{$A$ : Algebra Element}
    \index{$B$ : Algebra Element}
    \index{$n$ : Positive Integer}
    \index{$g$ : Group Element}

    Let $S$ be a commutative ring with $4S = 0$, $G$ a group and $A,B \in SG$ two commuting elements.
    Then for each positive integer $n$ we have \[(A + B)^{2^n} = A^{2^n} + B^{2^n} + 2A^{2^{n-1}}B^{2^{n-1}}.
    \]
    In particular, for $g \in G$ we have
    \[ (g-1)^{2^n} = (g^{2^n}-1) + 2(g^{2^{n-1}}-1). \]
\end{lem}
\begin{proof}
    \index{$i$ : Nonnegative Integer}
    \index{$\nu$ : Function}

    It is enough to show that $\binom{2^n}{0} = \binom{2^n}{2^n} \equiv 1 \mod 4$, $\binom{2^n}{2^{n-1}} \equiv 2 \mod 4$ and $\binom{2^n}{i} \equiv 0 \mod 4$ otherwise.
    Of course, $\binom{2^n}{0} = \binom{2^n}{2^n} = 1$.
    Let $\nu$ denote the $2$-adic valuation.
    By \cref{lem:KummerBinomial} we have $\nu\left(\binom{2^n}{i}\right) = n-\nu(i)$ for every $0 < i < 2^n$.
    So $\binom{2^n}{i} \equiv 0 \mod 4$ for $i \neq 2^{n - 1}$ and $\binom{2^n}{2^{n - 1}} \equiv 2 \mod 4$.
\end{proof}

\begin{lem}\label{lem:cyclic}
    \index{$S$ : Commutative Ring}
    \index{$\idealn$ : Ideal of Ring}
    \index{$n$ : Positive Integer}
    \index{$C_{2^n}$ : Finite Group}
    \index{$u$ : Group Element}
    \index{$U$ : Algebra Element}
    \index{$k$ : Nonnegative Integer}
    \index{$i$ : Nonnegative Integer}
    \index{$j$ : Integer}

    Let $S$ be a commutative ring with an ideal $\idealn$ and $4S = 0$.
    Let $n$ be a positive integer and $C_{2^n} = \gen{ u \given u^{2^n} = 1 }$.
    Set $U = u - 1 \in SC_{2^n}$.
    Then for each non-negative integer $k$,
    \begin{equation*}
        \Delta(SC_{2^n} : \idealn)^k = \bigoplus_{0 \leq i < 2^{n - 1}} \idealn^{k - i}U^i \oplus \bigoplus_{2^{n - 1} \leq i < 2^n} \left(\idealn^{k - i} + 2\idealn^{k - i - 2^{n - 1}}\right)U^i
    \end{equation*}
    where $\idealn^j = S$ for $j \leq 0$.
\end{lem}
\begin{proof}
    \index{$\Xi_k$ : Ideal of Algebra}
    \index{$\Theta$ : Ideal of Algebra}
    \index{$A$ : Algebra Element}
    \index{$B$ : Algebra Element}
    \index{$\alpha_i$ : Ring Element}
    \index{$\beta_i$ : Ring Element}
    \index{$\xi_i$ : Ring Element}
    \index{$\eta_i$ : Ring Element}

    Let $\Xi_k$ denote the right-hand side of the equation.
    Set $\Theta = \Delta(SC_{2^n} : \idealn)$.

    First, $\Theta^k \supseteq \idealn^{k - i}U^i$ for $0 \leq i < 2^n$ is clear.
    As $4S = 0$, \cref{lem:KummerMod4} shows that
    \begin{equation*}
        U^{2^n} = u^{2^n} + 2u^{2^{n - 1}} + 1 = 2U^{2^{n - 1}}.
    \end{equation*}
    Hence $2U^{2^{n - 1}} \in \Theta^{2^n}$ and $\Theta^k \supseteq 2\idealn^{k - i - 2^{n - 1}}U^i$ for $2^{n - 1} \leq i < 2^n$.
    Thus $\Theta^k \supseteq \Xi_k$.

    Next, we show that $\Theta^k \subseteq \Xi_k$ by induction on $k \geq 0$.
    The base case is evident.
    Let $A \in \Theta^k$ and $B \in \Theta$.
    By the inductive hypothesis, we can write
    \begin{align*}
        A &= \sum_{0 \leq i < 2^{n - 1}} \alpha_i U^i + \sum_{2^{n - 1} \leq i < 2^n} (\alpha_i + 2\xi_i)U^i \\
        B &= \sum_{0 \leq i < 2^{n - 1}} \beta_i U^i + \sum_{2^{n - 1} \leq i < 2^n} (\beta_i + 2\eta_i) U^i
    \end{align*}
    with some $\alpha_i \in \idealn^{k - i}$, $\beta_i \in \idealn^{1 - i}$ for $0 \leq i < 2^n$ and $\xi_i \in \idealn^{k - i - 2^{n - 1}}$, $\eta_i \in \idealn^{1 - i - 2^{n - 1}}$ for $2^{n - 1} \leq i < 2^n$.
    Then
    \begin{align*}
        AB
         &= \mathop{\underset{i + j < 2^{n - 1} + 2^n}{\sum_{0 \leq i < 2^n} \sum_{0 \leq j < 2^n}}} \alpha_i \beta_j U^{i + j}
        + \mathop{\underset{i + j \geq 2^{n - 1} + 2^n}{\sum_{0 \leq i < 2^n} \sum_{0 \leq j < 2^n}}} \alpha_i \beta_j U^{i + j}       \\
         &\qquad + \mathop{\underset{i + j < 2^n}{\sum_{0 \leq i < 2^n} \sum_{2^{n - 1} \leq j < 2^n}}} 2\alpha_i \eta_j U^{i + j}
        + \mathop{\underset{i + j \geq 2^n}{\sum_{0 \leq i < 2^n} \sum_{2^{n - 1} \leq j < 2^n}}} 2\alpha_i \eta_j U^{i + j}           \\
         &\qquad\qquad + \mathop{\underset{i + j < 2^n}{\sum_{2^{n - 1} \leq i < 2^n} \sum_{0 \leq j < 2^n}}} 2\xi_i \beta_j U^{i + j}
        + \mathop{\underset{i + j \geq 2^n}{\sum_{2^{n - 1} \leq i < 2^n} \sum_{0 \leq j < 2^n}}} 2\xi_i \beta_j U^{i + j}
    \end{align*}
    with $\alpha_i \beta_j \in \idealn^{k + 1 - (i + j)}$ and $\alpha_i \eta_j, \xi_i \beta_j \in \idealn^{k + 1 - (i + j) - 2^{n - 1}}$.
    The fourth and sixth double sums vanish as $2U^{2^n} = 4U^{2^{n - 1}} = 0$.
    The second double sum also vanishes as $U^{2^{n - 1} + 2^n} = 2U^{2^n} = 4U^{2^{n - 1}} = 0$.
    Thus
    \begin{align*}
        AB &= \mathop{\underset{i + j < 2^n}{\sum_{0 \leq i < 2^n} \sum_{0 \leq j < 2^n}}} \alpha_i \beta_j U^{i + j} + \mathop{\underset{2^n \leq i + j < 2^{n - 1} + 2^n}{\sum_{0 \leq i < 2^n} \sum_{0 \leq j < 2^n}}} \alpha_i \beta_j U^{i + j}  \\
           &\qquad + \mathop{\underset{i + j < 2^n}{\sum_{0 \leq i < 2^n} \sum_{2^{n - 1} \leq j < 2^n}}} 2\alpha_i \eta_j U^{i + j} + \mathop{\underset{i + j < 2^n}{\sum_{2^{n - 1} \leq i < 2^n} \sum_{0 \leq j < 2^n}}} 2\xi_i \beta_j U^{i + j}.
    \end{align*}
    In the second double sum, each term can be rewritten as
    \begin{equation*}
        \alpha_i \beta_j U^{i + j} = 2\alpha_i \beta_j U^{i + j - 2^{n - 1}}
    \end{equation*}
    and its coefficient satisfies
    \begin{equation*}
        \alpha_i \beta_j \in \idealn^{k + 1 - (i + j)}
        = \idealn^{k + 1 - (i + j - 2^{n - 1}) - 2^{n - 1}}.
    \end{equation*}
    Thus $AB \in \Xi_{k + 1}$.
\end{proof}

The above equations agree with those that follow from Jennings' theory when $S$ is a field of characteristic~$2$.
The next example illustrates the previous lemma explicitly for the cyclic group of order $4$.
\begin{exm}
    \index{$S$ : Commutative Ring}
    \index{$\idealn$ : Ideal of Ring}
    \index{$C_4$ : Finite Group}
    \index{$u$ : Group Element}
    \index{$U$ : Algebra Element}
    \index{$\Theta$ : Ideal of Algebra}

    Let $S$ be a commutative ring with an ideal $\idealn$ and $4S = 0$.
    Let $C_4 = \gen{ u \given u^4 = 1 }$.
    Set $U = u - 1 \in SC_4$ and $\Theta = \Delta(SC_4 : \idealn)$.
    Then \cref{lem:cyclic} shows the following.
    \begin{alignat*}{4}
        \Theta^0 &= S         & &\oplus SU         & &\oplus SU^2                      & &\oplus SU^3                \\
        \Theta^1 &= \idealn   & &\oplus SU         & &\oplus SU^2                      & &\oplus SU^3                \\
        \Theta^2 &= \idealn^2 & &\oplus \idealn U  & &\oplus SU^2                      & &\oplus SU^3                \\
        \Theta^3 &= \idealn^3 & &\oplus \idealn^2U & &\oplus \idealn U^2               & &\oplus SU^3                \\
        \Theta^4 &= \idealn^4 & &\oplus \idealn^3U & &\oplus (\idealn^2 + 2S)U^2       & &\oplus \idealn U^3         \\
        \Theta^5 &= \idealn^5 & &\oplus \idealn^4U & &\oplus (\idealn^3 + 2\idealn)U^2 & &\oplus (\idealn^2 + 2S)U^3 \\
                 &\vdots      & &                  & &                                 & &
    \end{alignat*}
    Accordingly the successive quotients as $S/\idealn$-modules are the following.
    \begin{alignat*}{4}
        \Theta^0/\Theta^1 &\cong S/\idealn           & &                           & &                                               & &                                \\
        \Theta^1/\Theta^2 &\cong \idealn/\idealn^2   & &\oplus S/\idealn           & &                                               & &                                \\
        \Theta^2/\Theta^3 &\cong \idealn^2/\idealn^3 & &\oplus \idealn/\idealn^2   & &\oplus S/\idealn                               & &                                \\
        \Theta^3/\Theta^4 &\cong \idealn^3/\idealn^4 & &\oplus \idealn^2/\idealn^3 & &\oplus \idealn/(\idealn^2 + 2S)                & &\oplus S/\idealn                \\
        \Theta^4/\Theta^5 &\cong \idealn^4/\idealn^5 & &\oplus \idealn^3/\idealn^4 & &\oplus (\idealn^2 + 2S)/(\idealn^3 + 2\idealn) & &\oplus \idealn/(\idealn^2 + 2S) \\
                          &\vdots                    & &                           & &                                               & &
    \end{alignat*}
\end{exm}

\begin{lem}\label{lem:dihedral}
    \index{$S$ : Commutative Ring}
    \index{$\idealn$ : Ideal of Ring}
    \index{$D_{16}$ : Finite Group}
    \index{$u$ : Group Element}
    \index{$v$ : Group Element}
    \index{$w$ : Group Element}
    \index{$U$ : Algebra Element}
    \index{$V$ : Algebra Element}
    \index{$W$ : Algebra Element}
    \index{$\mathscr{B}$ : Set}
    \index{$\omega$ : Function}
    \index{$s$ : Nonnegative Integer}
    \index{$t$ : Nonnegative Integer}
    \index{$i$ : Nonnegative Integer}
    \index{$k$ : Nonnegative Integer}
    \index{$Q$ : Algebra Element}
    \index{$j$ : Integer}

    Let $S$ be a commutative ring with an ideal $\idealn$ containing $2$ and $4S = 0$.
    Let
    \begin{equation*}
        D_{16} = \gen{ u, v, w \given u^2 = 1, v^2 = 1, w^4 = 1, [v, u] = w, [w, u] = w^2, [w, v] = w^2 }
    \end{equation*}
    be the dihedral group of order $16$.
    Set $U = u - 1, V = v - 1, W = w - 1 \in SD_{16}$.
    Moreover, let $\mathscr{B} = \set{ U^s V^t W^i \given 0 \leq s, t, i < 2 }$ and define a function $\omega\colon
        \mathscr{B} \rightarrow \ZZ$ by
    \begin{equation*}
        \omega(U^s V^t W^i) = s + t + 2i \qquad (0 \leq s, t, i < 2).
    \end{equation*}
    Then for each non-negative integer $k$,
    \begin{align*}
        \Delta(SD_{16} : \idealn)^k
         &= \bigoplus_{Q \in \mathscr{B}} \idealn^{k-\omega(Q)}Q \oplus \bigoplus_{Q \in \mathscr{B}} \left(\idealn^{k-4-\omega(Q)} + 2\idealn^{k-8-\omega(Q)} \right)QW^2
    \end{align*}
    where $\idealn^j = S$ for $j \leq 0$.
\end{lem}
\begin{proof}
    \index{$\Theta$ : Ideal of Algebra}
    \index{$z$ : Group Element}
    \index{$Z$ : Algebra Element}
    \index{$\Xi_k$ : Ideal of Algebra}
    \index{$j$ : Nonnegative Integer}
    \index{$A$ : Algebra Element}
    \index{$\alpha_Q$ : Ring Element}
    \index{$\beta_Q$ : Ring Element}
    \index{$\gamma_Q$ : Ring Element}

    Let $\Xi_k$ denote the right-hand side of the equation.
    Set $\Theta = \Delta(SD_{16} : \idealn)$.
    As an additional notation, set $z = w^2 = [w, u] = [w, v]$ and $Z = z-1 \in SD_{16}$.
    Then, by \cref{lem:KummerMod4}, we have $W^2 = Z + 2W$.
    (Though this is not strictly necessary, we note that $Z \in \Theta^3$ by \cref{lem:DimensionSubgroupsAndGamma}.)

    We will argue by induction on $k$.
    Since every element of $D_{16}$ can be written as $u^s v^t w^i z^j$ ($0 \leq s, t, i, j < 2$), the base case is clear:
    \begin{equation*}
        SD_{16} = \bigoplus_{Q \in \mathscr{B}} Q \oplus \bigoplus_{Q \in \mathscr{B}} QW^2.
    \end{equation*}

    So assume the lemma holds for all non-negative integers at most $k$.
    Note that $U,V \in \Theta$, $UV, W \in \Theta^2$, $UW, VW \in \Theta^3$ and $UVW, W^2 \in \Theta^4$.
    Thus, it is clear that $\idealn^{(k+1)-\omega(Q)}Q \subseteq \Theta^{k+1}$ and $\idealn^{(k+1)-4-\omega(Q)} QW^2 \subseteq \Theta^{k+1}$ for each $Q \in \mathscr{B}$.
    Moreover, to see that $\Xi_{k+1} \subseteq \Theta^{k+1}$ we need to show that $2W^2 \in \Theta^8$.
    This follows, using $Z^2 = 2Z$ as $z^2=1$, from \[(W^2)^2 = (Z+2W)^2 = Z^2 = 2Z = 2(W^2 + 2W) = 2W^2, \] as the first element of this equation is contained in $\Theta^8$.
    Hence we establish $\Xi_{k + 1} \subseteq \Theta^{k + 1}$.

    Also, we obtain that $\Theta/\Theta^2$ is generated by the images of $\idealn$, $U$ and $V$.
    Hence to show $\Theta^{k+1} \subseteq \Xi_{k+1}$ it is sufficient to show that for $A$ a generic element in $\Theta^k$ all of $\idealn A$, $UA$, $VA$ are contained in $\Xi_{k+1}$.

    Let, by the induction hypothesis, \[A = \sum_{Q \in \mathscr{B}} \alpha_Q Q + \sum_{Q \in \mathscr{B}} \left(\beta_Q + 2\gamma_Q \right)QW^2 \] for $\alpha_Q \in \idealn^{k-\omega(Q)}$, $\beta_Q \in \idealn^{k - 4 - \omega(Q) }$ and $\gamma_Q \in \idealn^{k-8-\omega(Q)}$ for all $Q \in \mathscr{B}$.
    Clearly, $\idealn A \subseteq \Xi_{k+1}$.
    Next, using the equation $U^2 = 2U$, which holds as $u^2=1$ and $2 = -2$ in $S$, we obtain
    \begin{align*}
        U \sum_{Q \in \mathscr{B}} \alpha_Q Q &= \sum_{0 \leq t, i < 2} (\alpha_{V^t W^i} + 2\alpha_{U V^t W^i}) U V^t W^i \intertext{ which is contained in $\Xi_{k+1}$ as $2 \in \idealn$ by assumption.
            For the same reasons
        }
        U \sum_{Q \in \mathscr{B}} \beta_Q QW^2
                                              &= \sum_{0 \leq t, i < 2} (\beta_{V^t W^i} + 2\beta_{U V^t W^i}) U V^t W^i W^2,                                                                              \\
        U \sum_{Q \in \mathscr{B}} 2\gamma_Q QW^2
                                              &= \sum_{0 \leq t, i < 2} 2\gamma_{V^t W^i} U V^t W^i W^2
    \end{align*}

    and we get $U\sum_{Q \in \mathscr{B}} \beta_Q QW^2, U\sum_{Q \in \mathscr{B}} 2\gamma_Q QW^2 \in \Xi_{k+1}$, so that $UA \in \Xi_{k+1}$.

    Next, consider $VA$.
    We have $V^2 = 2V$ as before.
    The calculations here are more evolved and we separate $A$ according to the summands $Q \in \mathscr{B}$.
    When $Q$ does not have $U$ as a factor, a short computation yields the following.
    \begin{align*}
        V (\alpha_{1}1 + (\beta_{1} + 2\gamma_{1})W^2)       &= \alpha_{1} V + (\beta_{1} + 2\gamma_{1}) VW^2,   \\
        V (\alpha_{V}V + (\beta_{V} + 2\gamma_{V})VW^2)      &= 2\alpha_{V} V + 2\beta_{V} VW^2,                 \\
        V (\alpha_{W}W + (\beta_{W} + 2\gamma_{W})WW^2)      &= \alpha_{W} VW + (\beta_{W} + 2\gamma_{W}) VWW^2, \\
        V (\alpha_{VW}VW + (\beta_{VW} + 2\gamma_{VW})VWW^2) &= 2\alpha_{VW} VW + 2\beta_{VW} VWW^2.
    \end{align*}
    When $Q$ does have $U$ as a factor, by elementary commutator formulas
    \begin{align*}
        VU = UV + (1 + U + V + UV)W,
        \qquad
        WV = VW + (1 + V + W + VW)Z
    \end{align*}
    and the power formulas $Z^2 = 2Z$, $Z = W^2 + 2W$, a lengthy computation yields the following.
    \begin{align*}
         &V (\alpha_{U}
        U + (\beta_{U} + 2\gamma_{U})UW^2)                                                    \\
         &\qquad = \alpha_{U}(UV + W + UW + VW + UVW)                                         \\
         &\qquad\qquad +	(\beta_{U} + 2\gamma_{U})(UV + W + UW + VW + UVW)W^2,                \\
         &V (\alpha_{UV}UV + (\beta_{UV} + 2\gamma_{UV})UVW^2)                                \\
         &\qquad = 2\alpha_{UV} (UV + W + UW) - \alpha_{UV} (VW + UVW)                        \\
         &\qquad\qquad -	(\alpha_{UV}+2\beta_{UV}) (1 + U)W^2 +	2\beta_{UV}	UVW^2             \\
         &\qquad\qquad +	\alpha_{UV}	(W + UW)W^2 - (\beta_{UV} + 2\gamma_{UV}) (VW + UVW)W^2, \\
         &V (\alpha_{UW}UW + (\beta_{UW} + 2\gamma_{UW})UWW^2)                                \\
         &\qquad = \alpha_{UW} UVW +	(\alpha_{UW} + 2\beta_{UW})	(1 + U + V + UV)W^2          \\
         &\qquad\qquad +	(\beta_{UW} + 2\gamma_{UW}) UVWW^2,                                  \\
         &V (\alpha_{UVW}UVW + (\beta_{UVW} + 2\gamma_{UVW})UVWW^2)                           \\
         &\qquad = 2\alpha_{UVW}	UVW - (\alpha_{UVW} + 2\beta_{UVW})(V + UV + W + UW)W^2      \\
         &\qquad\qquad + 2\beta_{UVW} UVWW^2.
    \end{align*}
    From these formulas, we conclude that $VA \in \Xi_{k + 1}$.

    Overall $\Theta^{k + 1} = \Xi_{k + 1}$.
\end{proof}

\section{Proof of main theorem}\label{sec:Approximation}

This section is devoted to proving \cref{main:A}.
Throughout this section, we assume that $n > m > \ell \ge 2$ and $G$ and $H$ are the groups of order $2^{n + m + \ell}$
\begin{alignat*}{8}
     &G &\  &= &\ \langle\, x, y, z \mid x^{2^n} &= 1, &\ y^{2^m} &= 1,       &\ z^{2^\ell} &= 1, &\ [y, x] &= z, &\ [z, x] &= z^{-2}, &\ [z, y] &= z^{-2} \,\rangle, \\
     &H &\  &= &\ \langle\, a, b, c \mid a^{2^n} &= 1, &\ b^{2^m} &= a^{2^m}, &\ c^{2^\ell} &= 1, &\ [b, a] &= c, &\ [c, a] &= c^{-2}, &\ [c, b] &= c^{-2} \,\rangle,
\end{alignat*}
which are counterexamples to the Modular Isomorphism Problem.
At the time of writing, these groups are essentially the only known examples.
We fix not only the groups $G$ and $H$ but also the generators $x$, $y$, $z$ and $a$, $b$, $c$ that are used in the presentations.

Moreover, we fix a commutative ring $R$ with an ideal $\idealr$ and a maximal ideal $\idealm$ such that $\idealm^2 \subseteq \idealr \subseteq \idealm$ and $R/\idealr$ has characteristic~$4$.
By \cref{lem:reduce}, it suffices to work over a quotient ring $S = R/\idealr$ to prove \cref{main:A}.
We fix the commutative ring $S$ of characteristic~$4$ for the rest of this section.
Let $\idealn = \idealm/\idealr$ and observe that it is a maximal ideal of $S$ with $2 \in \idealn$ and $2 \notin \idealn^2$.

Set $\Theta = \Delta(SH : \idealn) = \idealn \oplus \Delta(SH)$.
Note that the successive quotients $\Theta/\Theta^2$, $\Theta^2/\Theta^3$, \ldots have the canonical vector space structure over the residue field $F = S/\idealn$, as $\idealn \Theta^{k} \subseteq \Theta^{k+1}$ for each non-negative integer $k$.
We write $A = a - 1$, $B = b - 1$, $C = c - 1 \in SH$ and $\alpha$, $\beta$ for scalars in $F$.

\begin{hyp}\label{hyp:common}
    \index{$n$ : Positive Integer}
    \index{$m$ : Positive Integer}
    \index{$\ell$ : Positive Integer}
    \index{$G$ : Finite Group}
    \index{$x$ : Group Element}
    \index{$y$ : Group Element}
    \index{$z$ : Group Element}
    \index{$H$ : Finite Group}
    \index{$a$ : Group Element}
    \index{$b$ : Group Element}
    \index{$c$ : Group Element}
    \index{$\idealn$ : Ideal of Ring}
    \index{$A$ : Algebra Element}
    \index{$B$ : Algebra Element}
    \index{$C$ : Algebra Element}
    \index{$\Theta$ : Ideal of Algebra}
    \index{$\alpha$ : Field Element}
    \index{$\beta$ : Field Element}
    \index{$F$ : Field}

    For the convenience of the readers we summarize the common assumptions that will be used throughout.
    \begin{itemize}
        \item $n > m > \ell \geq 2$.
        \item The groups $G = \langle x, y, z \rangle$, $H = \langle a, b, c \rangle$ are defined as above.
        \item The commutative ring $S$ has characteristic~$4$.
        \item The maximal ideal $\idealn$ of $S$ satisfies $2 \in \idealn$ and $2 \notin \idealn^2$.
        \item $\Theta = \Delta(SH : \idealn) \subseteq SH$.
        \item $A = a - 1$, $B = b - 1$, $C = c - 1 \in SH$.
        \item $\alpha, \beta \in F = S/\idealn$.
    \end{itemize}
\end{hyp}

We will prove the following after several preliminary observations.
\begin{prop}\label{prop:char4}
    Under \cref{hyp:common}, $SG \not\cong SH$.
\end{prop}

We first record some elementary facts about the groups $G$ and $H$.

\begin{lem}\label{lem:x2y2}
    Under \cref{hyp:common}, $x^2, y^2 \in Z(G)$ and $a^2, b^2 \in Z(H)$.
\end{lem}

\begin{lem}\label{lem:ThetaModTheta2}
    Under \cref{hyp:common}, $\Theta/\Theta^2 \cong \idealn/\idealn^2 \oplus FA \oplus FB$.
\end{lem}
\begin{proof}
    This follows from \cref{prop:RelAug,prop:JenningsTheory}.
\end{proof}

First, we need to know how to approximate powers and commutators.

\begin{lem}\label{lem:commBA}
    Under \cref{hyp:common}, $[B, A] \equiv C \mod \Theta^3$.
\end{lem}
\begin{proof}
    Note that $[B, A] = (1 + A + B + AB)C$.
    \cref{lem:DimensionSubgroupsAndGamma} shows that $C \in \Theta^2$, and the lemma follows.
\end{proof}

\begin{lem}\label{lem:square}
    Under \cref{hyp:common}, $(\alpha A + \beta B)^2 \equiv \alpha^2 A^2 + \beta^2 B^2 + \alpha\beta C \mod \Theta^3$.
\end{lem}
\begin{proof}
    \begin{align*}
        (\alpha A + \beta B)^2
         &\equiv \alpha^2 A^2 + \beta^2 B^2 + \alpha\beta [B, A] \mod \Theta^3 & &(\text{by $2AB \in \Theta^3$}) \\
         &\equiv \alpha^2 A^2 + \beta^2 B^2 + \alpha\beta C \mod \Theta^3      & &(\text{by \cref{lem:commBA}}).
        \qedhere
    \end{align*}
\end{proof}

\begin{lem}\label{lem:commCACB}
    Under \cref{hyp:common}, $[C, A] \equiv [C, B] \equiv 2C \mod \Theta^4$.
\end{lem}
\begin{proof}
    \index{$d$ : Group Element}
    \index{$D$ : Algebra Element}

    Because the same proof works for $B$, we give the proof only for $A$.
    Set $d = [c, a]$ and $D = d - 1 \in SH$.
    Since $C \in \Theta^2$ and $D \in \Theta^3$ by \cref{lem:DimensionSubgroupsAndGamma},
    \begin{align*}
        [C, A]
         &= (1 + A + C + AC)D \equiv D                           & &(\text{by $D \in \Theta^3$})  \\
         &= (1 + C)^{-2} - 1                                     & &(\text{by $d = c^{-2}$})      \\
         &= (1 - 2C + 3C^2 - \dotsb) - 1 \equiv 2C \mod \Theta^4 & &(\text{by $C \in \Theta^2$}).
        \qedhere
    \end{align*}
\end{proof}

\begin{lem}\label{lem:modZSH}
    Under \cref{hyp:common}, $A^2 \equiv 2A$ and $B^2 \equiv 2B \mod Z(SH)$.
\end{lem}
\begin{proof}
    Note that $A^2 +2A = a^2-1$ and $B^2+2B = b^2-1$.
    Since $a^2$ and $b^2$ are central in $H$ by \cref{lem:x2y2}, $A^2 + 2A$ and $B^2 + 2B$ are central in $SH$.
\end{proof}

\begin{lem}\label{lem:fourth}
    Under \cref{hyp:common},
    \begin{equation*}
        (\alpha A + \beta B)^4 \equiv \alpha^4 A^4 + \beta^4 B^4 + \alpha^2\beta^2 C^2 \mod \Theta^5.
    \end{equation*}
\end{lem}
\begin{proof}
    Recall that $C \in \Theta^2$ by \cref{lem:DimensionSubgroupsAndGamma}.
    First \cref{lem:square} shows that
    \begin{align*}
        (\alpha A + \beta B)^4
         &\equiv (\alpha^2 A^2 + \beta^2 B^2 + \alpha\beta C)^2 \mod \Theta^5                                      \\
        \intertext{
            and, in the same way, we get
        }
        (\alpha A + \beta B)^4
         &\equiv \alpha^4 A^4 + \beta^4 B^4 + \alpha^2\beta^2 C^2                                                  \\
         &\qquad
        + \alpha^2\beta^2 [B^2, A^2] + \alpha^3\beta [C, A^2] + \alpha\beta^3 [C, B^2]
        \mod \Theta^5.                                                                                             \\
        \intertext{
            Applying \cref{lem:modZSH} yields
        }
        (\alpha A + \beta B)^4
         &\equiv \alpha^4 A^4 + \beta^4 B^4 + \alpha^2\beta^2 C^2 + 2\alpha^3\beta [C, A]  + 2\alpha\beta^3 [C, B]
        \mod \Theta^5                                                                                              \\
        \intertext{
            and, by \cref{lem:commCACB}, we conclude that
        }
        (\alpha A + \beta B)^4
         &\equiv \alpha^4 A^4 + \beta^4 B^4 + \alpha^2\beta^2 C^2
        \mod \Theta^5.
        \qedhere
    \end{align*}
\end{proof}

\begin{lem}\label{lem:power}
    Under \cref{hyp:common},
    \begin{equation*}
        (\alpha A + \beta B)^{2^{m + 1}} \equiv (\alpha + \beta)^{2^{m + 1}}
        A^{2^{m + 1}} \mod \Theta^{1 + 2^{m + 1}}.
    \end{equation*}
\end{lem}
\begin{proof}
    First applying \cref{lem:fourth} yields
    \begin{equation*}
        (\alpha A + \beta B)^{2^{m + 1}}
        \equiv (\alpha^4 A^4 + \beta^4 B^4 + \alpha^2\beta^2 C^2)^{2^{m - 1}}
        \mod \Theta^{1 + 2^{m + 1}}.
    \end{equation*}
    To compute this further observe that $A^4$ is central in $SH$ as $A^4 = a^4 + 2a^2 + 1$ and $a^2$ is central in $H$ by \cref{lem:x2y2}.
    The same is true for $B^4$ as well.
    Hence it follows from \cref{lem:KummerMod4} that
    \begin{equation*}
        (\alpha A + \beta B)^{2^{m + 1}}
        \equiv \alpha^{2^{m + 1}} A^{2^{m + 1}} + \beta^{2^{m + 1}} B^{2^{m + 1}} + \alpha^{2^m}\beta^{2^m} C^{2^m}
        \mod \Theta^{1 + 2^{m + 1}},
    \end{equation*}
    where we used $2A^{2^m}
        B^{2^m} \equiv 2A^{2^m}C^{2^{m - 1}} \equiv 2B^{2^m}C^{2^{m - 1}} \equiv 0$ here, which also follows using \cref{lem:KummerMod4}.
    Now $a^{2^m} = b^{2^m}$ and \cref{lem:KummerMod4} show that
    \begin{equation*}
        A^{2^{m + 1}} = (a^{2^{m + 1}} - 1) - 2(a^{2^m} - 1) = (b^{2^{m + 1}} - 1) - 2(b^{2^m} - 1) = B^{2^{m + 1}}.
    \end{equation*}
    Similarly, $m > \ell$, $c^{2^\ell} = 1$ and \cref{lem:KummerMod4} show that
    \begin{equation*}
        C^{2^m} = (c^{2^m} - 1) - 2(c^{2^{m - 1}} - 1) = 0.
    \end{equation*}
    Combining these, we conclude that
    \begin{align*}
        (\alpha A + \beta B)^{2^{m + 1}}
                                                  &\equiv (\alpha^{2^{m + 1}}  + \beta^{2^{m + 1}}) A^{2^{m + 1}}                            &                                   &                                      \\
                                                  &\equiv (\alpha^{2^{m + 1}}  + 2\alpha^{2^m}\beta^{2^m} + \beta^{2^{m + 1}}) A^{2^{m + 1}} &                                   &(\text{by $2A^{2^{m + 1}} \equiv 0$}) \\
                                                  &\equiv (\alpha  + \beta)^{2^{m + 1}}
        A^{2^{m + 1}} \mod \Theta^{1 + 2^{m + 1}} &                                                                                          &(\text{by \cref{lem:KummerMod4}}).
        \qedhere
    \end{align*}
\end{proof}

\begin{lem}\label{lem:A2m+1}
    Under \cref{hyp:common}, $A^{2^{m + 1}} \not\equiv 0 \mod \Theta^{1 + 2^{m + 1}}$.
\end{lem}
\begin{proof}
    \index{$N$ : Finite Group}
    \index{$K$ : Finite Group}
    \index{$u$ : Group Element}
    \index{$U$ : Algebra Element}

    The basic idea of the proof is reducing the problem to the cyclic case.
    Consider the normal subgroup $N = \langle a^{2^{m + 1}}, a^{-1}b, c \rangle$ of $H$, let $K = \gen{ u \given u^{2^{m + 1}} = 1 }$ and set $U = u-1 \in SK$.
    Since $H/N \cong K$ and $SH/\Delta(SN)SH \cong SK$ by $n > m$, we have a surjective homomorphism $SH \to SK$ defined by $A \mapsto U$.
    As the homomorphism maps $\Delta(SH : \idealn)$ onto $\Delta(SK : \idealn)$, it suffices to prove that
    \begin{equation*}
        U^{2^{m + 1}} \not\equiv 0 \mod \Delta(SK : \idealn)^{1 + 2^{m + 1}}.
    \end{equation*}
    The coefficients of $U^{2^m}$ in the direct sum decomposition of $\Delta(SK : \idealn)^{1 + 2^{m + 1}}$ in \cref{lem:cyclic} are elements of $(\idealn^{1 + 2^m} + 2\idealn)$.
    By \cref{lem:KummerMod4}, we have $U^{2^{m + 1}} = 2U^{2^m}$ and as $2 \in \idealn$ and $2 \notin \idealn^2$, we get $2U^{2^m} \not\in (\idealn^{1 + 2^m} + 2\idealn)U^{2^m}$.
    Hence, $U^{2^{m + 1}} \not\equiv 0 \mod \Delta(SK : \idealn)^{1 + 2^{m + 1}}$.
\end{proof}

\begin{lem}\label{lem:2C}
    Under \cref{hyp:common}, $2C \not\equiv 0 \mod \Theta^4$.
\end{lem}
\begin{proof}
    \index{$N$ : Finite Group}
    \index{$K$ : Finite Group}
    \index{$u$ : Group Element}
    \index{$v$ : Group Element}
    \index{$w$ : Group Element}
    \index{$U$ : Algebra Element}
    \index{$V$ : Algebra Element}
    \index{$W$ : Algebra Element}

    The basic idea of the proof is reducing the problem to the dihedral case.
    Consider the normal subgroup $N = \langle a^2, b^2, c^4 \rangle$ of $H$ and let
    \begin{equation*}
        K = \gen{ u, v, w \given u^2 = 1, \ v^2 = 1, \ w^4 = 1, \ [v, u] = w, \ [w, u] = w^2, \ [w, v] = w^2 }.
    \end{equation*}
    Set $U = u - 1, V = v - 1, W = w - 1 \in SK$.
    Since $H/N \cong K$ and $SH/\Delta(SN)SH \cong SK$, we have a surjective homomorphism $SH \to SK$ defined by $A \mapsto U$, $B \mapsto V$ and $C \mapsto W$.
    As the homomorphism maps $\Delta(SH : \idealn)$ onto $\Delta(SK : \idealn)$, it suffices to prove that
    \begin{equation*}
        2W \not\equiv 0 \mod \Delta(SK : \idealn)^4.
    \end{equation*}
    By \cref{lem:dihedral} and $2W \notin \idealn^2W$ as $2 \notin \idealn^2$, we get $2W \not\equiv 0 \mod \Delta(SK : \idealn)^4$.
\end{proof}

\begin{proof}[Proof of \cref{prop:char4}]
    \index{$\psi$ : Algebra Homomorphism}
    \index{$Y$ : Algebra Element}
    \index{$\alpha$ : Field Element}
    \index{$\beta$ : Field Element}

    Suppose that there exists an algebra isomorphism $SG \to SH$ to obtain a contradiction.
    Then we may construct a normalized isomorphism $\psi\colon SG \to SH$ from this by \cite[p.~193]{PolcinoMiliesSehgal02}.
    In particular, it is compatible with the augmentation ideals.
    Set $Y = y - 1 \in SG$.
    By \cref{lem:ThetaModTheta2}, there are unique scalars $\alpha, \beta$ of $F$ such that
    \begin{alignat*}{2}
        \psi(Y) &\equiv \alpha A + \beta B & &\mod \Theta^2.
    \end{alignat*}
    Note that elements corresponding to $\idealn/\idealn^2$ are not needed in these approximations as $\psi$ is normalized.
    Since $\psi$ is an isomorphism of $S$-algebra and $Y \notin \Delta(SG : \idealn)^2$ by \cref{lem:ThetaModTheta2}, we have $(\alpha,\beta) \neq (0,0)$.
    For the rest of the proof we turn to showing the contrary.

    Recall the power relation $y^{2^m} = 1$ in $G$.
    Then $Y^{2^{m + 1}} = 0$ by \cref{lem:KummerMod4} and
    \begin{equation}\label{eq:f}
        0
        \equiv
        \psi(Y)^{2^{m + 1}}
        \equiv (\alpha + \beta)^{2^{m + 1}}
        A^{2^{m + 1}} \mod \Theta^{1 + 2^{m + 1}}
    \end{equation}
    by \cref{lem:power}.
    Also recall that $y^2$ is central in $G$ by \cref{lem:x2y2}.
    Then $Y^2 + 2Y = y^2 - 1$ is central in $SG$.
    As $\psi$ is an isomorphism, $\psi(Y)^2 + 2\psi(Y)$ must be central in $SH$.
    In particular, it must commute with $A$ and $B$.
    By \cref{lem:square},
    \begin{equation*}
        \psi(Y)^2 + 2\psi(Y) \equiv \alpha^2A^2 + \beta^2B^2 + \alpha\beta C + 2\alpha A + 2\beta B \mod \Theta^3.
    \end{equation*}
    So by \cref{lem:modZSH}, \[\psi(Y)^2 + 2\psi(Y) \equiv 2(\alpha + \alpha^2)A + 2(\beta + \beta^2)B + \alpha\beta C \mod Z(SH) + \Theta^3.
    \]
    Then it follows from \cref{lem:commBA,lem:commCACB} that
    \begin{alignat}{3}
        0 &\equiv [\psi(Y)^2 + 2\psi(Y), A] & &\equiv 2(\beta + \beta^2 + \alpha\beta)C   & &\mod \Theta^4, \label{eq:g} \\
        0 &\equiv [\psi(Y)^2 + 2\psi(Y), B] & &\equiv 2(\alpha + \alpha^2 + \alpha\beta)C & &\mod \Theta^4. \label{eq:h}
    \end{alignat}

    From \eqref{eq:f}, \eqref{eq:g} and \eqref{eq:h} we get
    \begin{equation*}
        \left\{
        \begin{aligned}
            \alpha + \beta = 0                \\
            \beta + \beta^2 + \alpha\beta = 0 \\
            \alpha + \alpha^2 + \alpha\beta = 0
        \end{aligned}
        \right.
    \end{equation*}
    by \cref{lem:A2m+1,lem:2C}.
    This gives $(\alpha, \beta) = (0, 0)$ and shows that $SG \not\cong SH$.
\end{proof}

\begin{proof}[Proof of \cref{main:A}]
    \index{$S$ : Commutative Ring}
    \index{$\idealn$ : Ideal of Ring}

    From the assumptions, the quotient ring $S = R/\idealr$ has characteristic~$4$ and $\idealn = \idealm/\idealr$ is a maximal ideal of $S$ with $2 \in \idealn$ and $2 \notin \idealn^2$.
    By~\cref{prop:char4}, $SG \not\cong SH$.
    Hence, $RG \not\cong RH$ by \cref{lem:reduce}.

    On the other hand, if $R$ has characteristic~$2$, then the proof of \cite[Theorem~3.1]{MargolisSakurai25} carries over verbatim just replacing the field in the proof by the ring $R$ which shows $RG \cong RH$ in this case.
\end{proof}

\section{Rational and complex group algebras}\label{sec:Rational}
\index{$T$ : Variable}
\index{$\nu$ : Function}
\index{$\totient$ : Function}
\index{$\Gamma$ : Finite Group}
\index{$\CS(\Gamma)$ : Set}

\cref{main:A} shows that $RG \cong RH$ when $R$ has characteristic~$2$.
It also shows that $RG \not\cong RH$ in many cases\footnote{Note that the ring $\ZZ[T]/(4, T^2 - 2)$ is excluded from \cref{main:A}, for instance.
} when the characteristic of $R$ is divisible by~$4$.
When $2$ is a unit in~$R$, the situation is more delicate as the following shows (cf.~\cite[Corollary~8.2]{Margolis22}).

\begin{prop}\label{prop:unit}
    \index{$G$ : Finite Group}
    \index{$H$ : Finite Group}

    Let $G$ and $H$ be the groups defined in \cref{main:A}.
    Then $\CC G \cong \CC H$, but $\QQ G \not\cong \QQ H$.
\end{prop}

A proof will be given after a series of lemmas.
Let $\nu$ denote the $2$-adic valuation on the integers and let $\totient$ denote Euler's totient function.
We write $\CS(\Gamma)$ for the set of conjugacy classes of cyclic subgroups of a finite group~$\Gamma$.
Then
\begin{equation*}
    |\CS(\Gamma)| = \sum_{g \in \Gamma} \frac{1}{|\Gamma : N_\Gamma(\langle g \rangle)|} \times \frac{1}{\totient(|g|)}
\end{equation*}
as $|\Gamma : N_\Gamma(\langle g \rangle)|$ equals the number of conjugates of $\langle g \rangle$ and $\totient(|g|)$ equals the number of generators of $\langle g \rangle$.
The following closed formula for the number of cyclic subgroups of a finite abelian $p$-group will be useful for us.

\begin{lem}\label{lem:count}
    \index{$p$ : Prime Number}
    \index{$n$ : Nonnegative Integer}
    \index{$m$ : Nonnegative Integer}
    \index{$\ell$ : Nonnegative Integer}

    Let $p$ be a prime and $n \ge m \ge \ell \ge 0$ integers.
    Then
    \begin{align*}
         &|\CS(C_{p^n} \times C_{p^m} \times C_{p^\ell})|                                                                                            \\
         &\qquad = (n - m)p^{m + \ell} + (p^{2\ell + 1} + p^{2\ell})\frac{p^{m - \ell} - 1}{p - 1} + (p^2 + p + 1)\frac{p^{2\ell} - 1}{p^2 - 1} + 1.
    \end{align*}
\end{lem}
\begin{proof}
    \index{$i$ : Nonnegative Integer}
    \index{$j$ : Nonnegative Integer}
    \index{$k$ : Nonnegative Integer}

    There are many known formulas to calculate the number of cyclic subgroups of a finite abelian group and we will use the one given in~\cite[Theorem~1]{Toth12}.
    In our case it reads
    \begin{equation*}
        |\CS(C_{p^n} \times C_{p^m} \times C_{p^\ell})| = \sum_{\substack{0 \le i \le n \\
        0 \le j \le m                                                                   \\
                0 \le k \le \ell}}
        \frac{\totient(p^i)\totient(p^j)\totient(p^k)}{\totient(\lcm(p^i, p^j, p^k))}.
    \end{equation*}
    To get the closed formula, we derive three difference equations first.
    \begin{equation*}
        \left\{
        \begin{aligned}
            |\CS(C_{p^{n + 1}} \times C_{p^m} \times C_{p^\ell})| - |\CS(C_{p^n} \times C_{p^m} \times C_{p^\ell})|                         & = p^{m + \ell},                              \\
            |\CS(C_{p^{m + 1}} \times C_{p^{m + 1}} \times C_{p^\ell})| - |\CS(C_{p^m} \times C_{p^m} \times C_{p^\ell})|                   & = p^{m + \ell + 1} + p^{m + \ell},           \\
            |\CS(C_{p^{\ell + 1}} \times C_{p^{\ell + 1}} \times C_{p^{\ell + 1}})| - |\CS(C_{p^\ell} \times C_{p^\ell} \times C_{p^\ell})| & = p^{2\ell + 2} + p^{2\ell + 1} + p^{2\ell}.
        \end{aligned}
        \right.
    \end{equation*}
    We only show the second one in details as similar arguments show the rest as well.
    Recall that $\totient(p^{m + 1}) = p^{m + 1} - p^m$ and $\sum_{0 \le k \le \ell} \totient(p^k) = p^\ell$.
    Then
    \begin{align*}
         &|\CS(C_{p^{m + 1}} \times C_{p^{m + 1}} \times C_{p^\ell})| - |\CS(C_{p^m} \times C_{p^m} \times C_{p^\ell})| \\
         &\qquad=
        \sum_{\substack{ 0 \le j \le m                                                                                  \\
                0 \le k \le \ell}} \totient(p^j)\totient(p^k)
        + \sum_{\substack{ 0 \le i \le m                                                                                \\
                0 \le k \le \ell}} \totient(p^i)\totient(p^k)
        - \totient(p^{m + 1})\sum_{0 \le k \le \ell} \totient(p^k)                                                      \\
         &\qquad= p^{m + \ell} + p^{m + \ell} - (p^{m + 1} - p^m)p^\ell                                                 \\
         &\qquad= p^{m + \ell + 1} + p^{m + \ell}.
    \end{align*}

    From these difference equations, we obtain
    \begin{align*}
         &|\CS(C_{p^n} \times C_{p^m} \times C_{p^\ell})|                                                                                                 \\
         &\qquad= (n - m)p^{m + \ell} + |\CS(C_{p^m} \times C_{p^m} \times C_{p^\ell})|                                                                   \\
         &\qquad= (n - m)p^{m + \ell} + (p^{2\ell + 1} + p^{2\ell})\frac{p^{m - \ell} - 1}{p - 1} + |\CS(C_{p^\ell} \times C_{p^\ell} \times C_{p^\ell})| \\
         &\qquad = (n - m)p^{m + \ell} + (p^{2\ell + 1} + p^{2\ell})\frac{p^{m - \ell} - 1}{p - 1} + (p^2 + p + 1)\frac{p^{2\ell} - 1}{p^2 - 1} + 1.
    \end{align*}
\end{proof}

\begin{lem}\label{lem:normalGH}
    \index{$\Gamma$ : Finite Group}
    \index{$G$ : Finite Group}
    \index{$H$ : Finite Group}
    \index{$K$ : Finite Group}

    Let $\Gamma$ be the group $G$ or $H$ defined in \cref{main:A} and $K$ a cyclic subgroup of $C_\Gamma(\Gamma')$.
    Then
    \begin{equation*}
        K \trianglelefteq \Gamma \iff \text{$K \leq Z(\Gamma)$ or $K \leq \Omega(\Gamma : \Gamma')$}.
    \end{equation*}
\end{lem}
\begin{proof}
    \index{$\alpha$ : Group Element}
    \index{$\beta$ : Group Element}
    \index{$\gamma$ : Group Element}
    \index{$\delta$ : Group Element}
    \index{$g$ : Group Element}
    \index{$i$ : Nonnegative Integer}
    \index{$j$ : Nonnegative Integer}
    \index{$k$ : Nonnegative Integer}
    \index{$t$ : Nonnegative Integer}

    Set $\alpha = x$, $\beta = y$ and $\gamma = z$, if $\Gamma = G$ and $\alpha = a$, $\beta = b$ and $\gamma = c$, if $\Gamma = H$.
    Observe that $C_\Gamma(\Gamma')$ is a maximal and abelian subgroup of $\Gamma$ generated by $\alpha\beta$, $\alpha^2$, $\beta^2$ and $\gamma$.
    Set $\delta = y^2$ if $\Gamma = G$ and $\delta = a^2$, if $\Gamma = H$.
    Then
    \begin{align*}
        C_\Gamma(\Gamma') &= \langle \alpha^{-1}\beta \rangle \times \langle \delta \rangle \times \langle \gamma \rangle.
    \end{align*}
    Moreover, $\Omega(\Gamma:\Gamma')$ is generated by $\gamma$ and the socle of $\langle \alpha^{-1}\beta, \delta \rangle$ which is also central in $\Gamma$.
    Hence, the second condition implies the first.
    So we assume $K \trianglelefteq \Gamma$ and $K \not\leq Z(\Gamma)$ and prove $K \leq \Omega(\Gamma : \Gamma')$.

    Let $g$ be a generator of $K$ and write it as
    \begin{equation*}
        g = (\alpha^{-1}\beta)^i \delta^j \gamma^k \qquad (0 \leq i < |\alpha^{-1}\beta|,\ 0 \leq j < |\delta|,\ 0 \leq k < |\gamma|).
    \end{equation*}
    The first assumption $K \trianglelefteq \Gamma$ implies $K^\alpha = K$.
    As $g^\alpha = g[g, \alpha]$, there is an integer $t \geq 0$ such that $[g, \alpha] = g^t$.
    Since $[g, \alpha] = \gamma^{i - 2k}$, we obtain
    \begin{equation*}
        (\alpha^{-1}\beta)^{it} = 1,
        \quad
        \delta^{jt} = 1,
        \quad
        \gamma^{kt} = \gamma^{i - 2k}
    \end{equation*}
    and accordingly
    \begin{equation}\label{eq:3ieqsG}
        \nu(it) \geq \nu(|\alpha^{-1}\beta|),
        \quad
        \nu(jt) \geq \nu(|\delta|),
        \quad
        \nu(kt - i + 2k) \geq \nu(|\gamma|).
    \end{equation}
    The second assumption $K \not\leq Z(\Gamma)$ implies $[g, \alpha] \neq 1$ as $\Gamma = \langle \alpha, C_\Gamma(\Gamma') \rangle$.
    So $\gamma^{kt} \neq 1$ and
    \begin{equation}\label{eq:1ieqG}
        \nu(|\gamma|) > \nu(kt).
    \end{equation}
    From \eqref{eq:1ieqG}, the first of \eqref{eq:3ieqsG} and $|\alpha^{-1}\beta| > |\gamma|$, we obtain $\nu(i) \geq \nu(|\alpha^{-1}\beta|) - \nu(t) > \nu(|\gamma|) - \nu(t) \geq \nu(2k)$, which implies $\nu(i - 2k) = \nu(2k)$.
    From \eqref{eq:1ieqG} and the third of \eqref{eq:3ieqsG}, we also obtain $\nu(kt - i + 2k) > \nu(kt)$, which implies $\nu(i - 2k) = \nu(kt - kt + i - 2k) = \nu(kt)$.
    These yield $\nu(2k) = \nu(kt)$, and hence $\nu(t) = 1$.
    Thus $\langle g^2 \rangle = \langle g^t \rangle = \langle [g, \alpha] \rangle$ and $K \leq \Omega(\Gamma : \Gamma')$.
\end{proof}

\begin{lem}\label{lem:diff}
    \index{$G$ : Finite Group}
    \index{$H$ : Finite Group}

    Let $G$ and $H$ be the groups defined in \cref{main:A}.
    Then
    \begin{equation*}
        |\CS(H)| - |\CS(G)| = (n-m)2^{m-1}(2^{\ell-1}-1).
    \end{equation*}
\end{lem}
\begin{proof}
    \index{$g$ : Group Element}
    \index{$i$ : Nonnegative Integer}
    \index{$j$ : Nonnegative Integer}
    \index{$k$ : Nonnegative Integer}
    \index{$t$ : Positive Integer}
    \index{$h$ : Group Element}

    We start the calculation, first for $G$.
    For better overview we record the conjugation action:
    \[y^x = yz, \ x^y = xz^{-1}, \ z^x = z^y = z^{-1}, \ x^z = xz^2, \ y^z = yz^2. \]

    \case{Inside maximal abelian subgroup of $G$}
    We first consider the elements in the maximal subgroup which is also abelian
    \begin{equation*}
        C_G(G') = \langle xy \rangle \times \langle y^2 \rangle \times \langle z \rangle \cong C_{2^n} \times C_{2^{m - 1}} \times C_{2^\ell}.
    \end{equation*}
    Every element in $C_G(G')$ has either class length $1$, when it lies in the center $Z(G)$, or $2$, as then the centralizer equals $C_G(G')$.
    By \cref{lem:normalGH}, we can count the number of conjugacy classes of cyclic subgroups of $G$ lying in $C_G(G')$ using the inclusion-exclusion principle as
    \begin{equation}\label{eq:CycGA}
        \begin{aligned}
             & \sum_{g \in C_G(G')} \frac{1}{|G : N_G(\langle g \rangle)|} \times \frac{1}{\totient(|g|)}                           \\
             & \qquad= \sum_{g \in Z(G) \cup \Omega(G : G')} \frac{1}{\totient(|g|)}
            + \frac{1}{2}\sum_{g \in C_G(G') \setminus (Z(G) \cup \Omega(G : G'))} \frac{1}{\totient(|g|)}                          \\
             & \qquad= \frac{1}{2}\big(|\CS(C_G(G'))| + |\CS(Z(G))| + |\CS(\Omega(G : G'))| - |\CS(\Omega(G : G') \cap Z(G))|\big).
        \end{aligned}
    \end{equation}

    \case{Outside maximal abelian subgroup of $G$}
    The elements of $G \setminus C_G(G') = x\Phi(G) \sqcup y\Phi(G)$ are of the form
    \begin{equation*}
        x(x^2)^i (y^2)^j z^k
        \quad \text{or} \quad
        y(x^2)^i (y^2)^j z^k
        \qquad
        (0 \leq i < 2^{n-1},\ 0 \leq j < 2^{m-1},\ 0 \leq k < 2^{\ell}).
    \end{equation*}
    Such an element $g$ satisfies $g^{yx} = gz$ which implies $g^G = gG'$.
    It also satisfies $z^{2^{\ell - 1}} \not\in \langle g \rangle$, so $G' \cap \langle g \rangle = 1$ and thus $N_G(\langle g \rangle) = C_G(g)$.
    For the elements of $x\Phi(G)$ note that they have a common order $2^n$.
    So the contribution is
    \begin{align*}
        \sum_{g \in x\Phi(G)} \frac{1}{|G : N_G(\langle g \rangle)|} \times \frac{1}{\totient(|g|)}
        = \frac{|x\Phi(G)|}{|G'| \times |(\ZZ/2^n\ZZ)^\times|}
        = \frac{2^{n - 1} \times 2^{m - 1} \times 2^\ell}{2^\ell \times 2^{n - 1}}
            = 2^{m - 1}.
    \end{align*}

    For the elements of $y\Phi(G)$ the calculation is more delicate as their orders are not uniform.
    Consider a chain of sets
    \begin{equation*}
        y\Omega_m(\Phi(G)) \subseteq y\Omega_{m + 1}(\Phi(G)) \subseteq \dotsb \subseteq y\Omega_{n - 1}(\Phi(G)) = y\Phi(G)
    \end{equation*}
    and observe that an element $g \in y\Omega_t(\Phi(G)) \setminus y\Omega_{t - 1}(\Phi(G))$ has order $2^t$ for each $m + 1 \leq t \leq n - 1$.
    Also an element $g \in y\Omega_m(\Phi(G))$ has order $2^m$.
    So the contribution is
    \begin{align*}
         &\sum_{g \in y\Phi(G)} \frac{1}{|G : N_G(\langle g \rangle)|} \times \frac{1}{\totient(|g|)}                                                                                                         \\
         &\qquad= \sum_{g \in y\Omega_m(\Phi(G))} \frac{1}{|G : N_G(\langle g \rangle)|} \times \frac{1}{\totient(|g|)}                                                                                       \\
         &\qquad\qquad + \sum_{t = m + 1}^{n - 1} \sum_{g \in y\Omega_t(\Phi(G)) \setminus y\Omega_{t - 1}(\Phi(G))} \frac{1}{|G : N_G(\langle g \rangle)|} \times \frac{1}{\totient(|g|)}                    \\
         &\qquad= \frac{|y\Omega_m(\Phi(G))|}{|G'| \times |(\ZZ/2^m\ZZ)^\times|} + \sum_{t = m + 1}^{n - 1} \frac{|y\Omega_t(\Phi(G)) \setminus y\Omega_{t - 1}(\Phi(G))|}{|G'| \times |(\ZZ/2^t\ZZ)^\times|} \\
         &\qquad= \frac{2^m \times 2^{m - 1} \times 2^\ell}{2^\ell \times 2^{m - 1}} + \sum_{t = m + 1}^{n - 1} \frac{2^{t - 1} \times 2^{m - 1} \times 2^\ell}{2^\ell \times 2^{t - 1}}                      \\
         &\qquad= 2^m + (n - m - 1)2^{m - 1}.
    \end{align*}
    The overall contribution of $G \setminus C_G(G')$ to $|\CS(G)|$ is
    \begin{equation}\label{eq:CycGNonA}
        \begin{aligned}
             & \sum_{g \in G \setminus C_G(G')} \frac{1}{|G : N_G(\langle g \rangle)|} \times \frac{1}{\totient(|g|)} \\
             & \qquad= \sum_{g \in x\Phi(G)} \frac{1}{|G : N_G(\langle g \rangle)|} \times \frac{1}{\totient(|g|)}
            + \sum_{g \in y\Phi(G)} \frac{1}{|G : N_G(\langle g \rangle)|} \times \frac{1}{\totient(|g|)}             \\
             & \qquad= 2^m + (n-m)2^{m-1}.
        \end{aligned}
    \end{equation}

    For the calculation in $H$ we also record the conjugacy action, which is essentially the same as for $G$: \[b^a = bc, \ a^b = ac^{-1}, \ c^a = c^b = c^{-1}, \ a^c = ac^2, \ b^c = bc^2.
    \]

    \case{Inside maximal abelian subgroup of $H$}
    The arguments are very similar to the case of the maximal abelian subgroup of $G$, so we make them shorter.
    First note that
    \begin{equation*}
        C_H(H') = \langle a^2 \rangle \times \langle a^{-1}b \rangle \times \langle c \rangle \cong C_{2^{n - 1}} \times C_{2^m} \times C_{2^\ell}
    \end{equation*}
    is a maximal subgroup of $H$ which is abelian.
    The class length of an element in $C_H(H')$ is $1$, if it is central in $H$, and $2$ otherwise.
    By \cref{lem:normalGH}, again by the inclusion-exclusion principle, we obtain the contribution to $|\CS(H)|$ as
    \begin{equation}\label{eq:CycHB}
        \begin{aligned}
             & \sum_{h \in C_H(H')} \frac{1}{|H : N_H(\langle h \rangle)|} \times \frac{1}{\totient(|h|)}                           \\
             & \qquad= \sum_{h \in Z(H) \cup \Omega(H : H')} \frac{1}{\totient(|h|)}
            + \frac{1}{2}\sum_{h \in C_H(H') \setminus (Z(H) \cup \Omega(H : H'))} \frac{1}{\totient(|h|)}                          \\
             & \qquad= \frac{1}{2}\big(|\CS(C_H(H'))| + |\CS(Z(H))| + |\CS(\Omega(H : H'))| - |\CS(\Omega(H : H') \cap Z(H))|\big).
        \end{aligned}
    \end{equation}

    \case{Outside maximal abelian subgroup of $H$}
    The elements in $H \setminus C_H(H') = aC_H(H')$ are of the form
    \begin{equation*}
        a(a^2)^i (a^{-1}b)^j c^k
        \qquad
        (0 \leq i < 2^{n-1},\ 0 \leq j < 2^m,\ 0 \leq k < 2^\ell).
    \end{equation*}
    Such an element $h$ satisfies $h^{ba} = hc$ which implies $h^H = hH'$ and its order equals $2^n$.
    It also satisfies $c^{2^{\ell - 1}} \not\in \langle h \rangle$, so $H' \cap \langle h \rangle = 1$ and thus $N_H(\langle h \rangle) = C_H(h)$.
    The overall contribution to $|\CS(H)|$ is
    \begin{equation}\label{eq:CycHNonB}
        \begin{aligned}
             & \sum_{h \in H \setminus C_H(H')} \frac{1}{|H : N_H(\langle h \rangle)|} \times \frac{1}{\totient(|h|)} \\
             & \qquad= \frac{|H \setminus C_H(H')|}{|H'| \times |(\ZZ/2^n\ZZ)^\times|}
            = \frac{2^{n - 1} \times 2^m \times 2^\ell}{2^\ell \times 2^{n - 1}}
            = 2^m.
        \end{aligned}
    \end{equation}

    \case{Comparing the values}
    We first note that
    \begin{alignat*}{2}
         &Z(G)           \cong Z(H)                               & &\cong C_{2^{n - 1}} \times C_{2^{m - 1}} \times C_2 \\
         &\Omega(G : G') \cong \Omega(H : H')                     & &\cong C_2 \times C_2 \times C_{2^\ell}              \\
         &\Omega(G : G') \cap Z(G) \cong \Omega(H : H') \cap Z(H) & &\cong C_2 \times C_2 \times C_2.
    \end{alignat*}
    So, comparing the expressions in \eqref{eq:CycGA} and \eqref{eq:CycHB}, the groups are pairwise isomorphic except for $C_G(G')$ and $C_H(H')$ themselves.
    Hence, the difference between the numbers of conjugacy classes of cyclic groups of $H$ and $G$, which equals $\eqref{eq:CycHB} + \eqref{eq:CycHNonB} - \eqref{eq:CycGA} - \eqref{eq:CycGNonA}$  is
    \begin{equation}\label{eq:CycFirstDiff}
        |\CS(H)| - |\CS(G)| = \frac{1}{2}\big(|\CS(C_H(H'))| - |\CS(C_G(G'))|\big) - (n-m)2^{m-1}.
    \end{equation}
    Recall that $C_H(H') \cong C_{2^{n-1}} \times C_{2^{m}} \times C_{2^\ell}$ and $C_G(G') \cong C_{2^n} \times C_{2^{m-1}} \times C_{2^\ell}$.
    It follows from \cref{lem:count} that
    \begin{align*}
         &|\CS(C_H(H'))| - |\CS(C_G(G'))|                                                                  \\
         &\qquad= ((n - m - 1)2^{m + \ell} + (2^{2\ell + 1} + 2^{2\ell})(2^{m - \ell} - 1))                \\
         &\qquad\qquad - ((n - m + 1)2^{m + \ell - 1} + (2^{2\ell + 1} + 2^{2\ell})(2^{m - \ell - 1} - 1)) \\
         &\qquad= (n - m)2^{m + \ell - 1}.
    \end{align*}
    Hence substituting $|\CS(C_H(H'))| - |\CS(C_G(G'))|$ in \eqref{eq:CycFirstDiff} gives
    \begin{equation*}
        |\CS(H)| - |\CS(G)| = (n-m)2^{m-1}(2^{\ell-1}-1).
    \end{equation*}
\end{proof}

\begin{proof}[Proof of \cref{prop:unit}]
    To see that $\CC G \cong \CC H$ note that both $G$ and $H$ posses a maximal subgroup which is abelian, namely $C_G(G')$ and $C_H(H')$, respectively.
    Hence by \cite[Theorem~12.11]{Isaacs76} the degrees of irreducible characters are either one or two.
    Then the Wedderburn decompositions of these algebras are
    \begin{align*}
        \CC G &\cong |G/G'|\CC \times \frac{1}{4}(|G|-|G/G'|)M_2(\CC), \\
        \CC H &\cong |H/H'|\CC \times \frac{1}{4}(|H|-|H/H'|)M_2(\CC)
    \end{align*}
    where $M_2(\CC)$ denotes the matrix algebra of degree~$2$ over~$\CC$.
    As $|G| = |H|$ and $|G/G'| = |H/H'|$ we have $\CC G \cong \CC H$.

    To see that $\QQ G \not\cong \QQ H$ note that the difference of the numbers of conjugacy classes of cyclic subgroups
    \begin{equation*}
        |\CS(H)| - |\CS(G)| = (n - m)2^{m - 1}(2^{\ell - 1} - 1)
    \end{equation*}
    is positive by \cref{lem:diff}, $n > m$ and $\ell \geq 2$.
    As the number of conjugacy classes of cyclic subgroups coincides with the number of Wedderburn components of the rational group algebra of a finite group \cite[Corollary~7.1.12 (Artin)]{JespersDelRio16} this shows that $\QQ G \not\cong \QQ H$.
\end{proof}

Brauer~\cite[Problem~$2^*$]{Brauer63} asked whether there are non-isomorphic finite groups that have isomorphic group algebras over every field.
Dade~\cite{Dade71} discovered an example among finite solvable groups whose orders are divisible by exactly two distinct primes.
Although our groups $G$ and $H$ have isomorphic group algebras over every field of characteristic~$2$ by \cref{main:A}, these groups can be distinguished over the rationals by \cref{prop:unit}.
We conclude by posing the following problem (see also \cite[Problem~8.3]{Margolis22}).
\begin{prob}
    \index{$p$ : Prime Number}

    Let $p$ be a prime.
    Are there non-isomorphic finite $p$-groups that have isomorphic group algebras over every field?
\end{prob}

\bibliographystyle{abbrv}
\bibliography{references}

%\listoffixmes

%\printindex % for proofreading

\end{document}